\RequirePackage{fix-cm}
\documentclass[smallextended]{svjour3}       % onecolumn (second format)
\smartqed  % flush right qed marks, e.g. at end of proof
\usepackage{graphicx}
\usepackage{amsmath}
\usepackage{amsfonts,amscd,amssymb,latexsym,amsbsy,url,calc,enumerate}

\DeclareMathOperator{\tr}{tr}

\DeclareMathOperator{\dsc}{dsc}
\DeclareMathOperator{\diam}{diam}
\DeclareMathOperator{\sgn}{sgn}
\newcommand{\mcp}{\mathcal P}

\newcommand{\NN}{\mathbb{N}}
\newcommand{\RR}{\mathbb{R}}

\begin{document}

\title{On a weighted generalization of Kendall's tau distance \thanks{The second author was supported by H2020-MSCA-RISE-2014, Project number 645672 (AMMODIT: ``Approximation Methods for Molecular Modelling and Diagnosis Tools'').}
}

\titlerunning{On a weighted generalization of Kendall's tau distance}        % if too long for running head

\author{Albert Bruno Piek         \and
        Evgeniy Petrov %etc.
}


\institute{Albert Bruno Piek \at
              Institute of Mathematics, University of L\"{u}beck, Ratzeburger Allee 160, 23562 L\"{u}beck,  Germany  \\
              \email{piek@math.uni-luebeck.de}           %  \\
           \and
           Evgeniy Petrov \at
              Institute of Applied Mathematics and Mechanics of the NAS of Ukraine,  Dobrovolskogo Str.~1, 84100 Slovyansk, Ukraine \\
              \email{eugeniy.petrov@gmail.com}
}


\maketitle

\begin{abstract}
We introduce a metric on the set of permutations of given order,  which is a weighted generalization of Kendall's $\tau$ rank distance and study its properties. Using the edge graph of a permutohedron, we give a criterion which guarantees that a permutation lies metrically between another two fixed permutations. In addition, the conditions under which four points from the resulting metric space form a pseudolinear quadruple were found.
\end{abstract}

\section{Introduction}

The concept of metric spaces appeared more than one century ago in works of Maurice Fr\'{e}chet and Felix Hausdorff. Recall that a {metric} on a set $X$ is a function $d\colon X\times X\to \mathbb{R}_+$, $\mathbb{R}_+=[0,\infty)$, such that the following conditions hold:

  (i) $d(x,y)=d(y,x)$ (symmetry),

  (ii) $(d(x,y)=0)\Leftrightarrow (x=y)$ (identity of indiscernibles),

  (iii) $d(x,y)\leqslant d(x,z)+d(z,y)$ (triangle inequality),

\noindent for all $x, y, z \in X$. The pair $(X,d)$ is called a {metric space}.

It is possible to define metrics not only on the ``abstract'' sets of points but also on the sets of various mathematical objects. The well-known Encyclopedia of Distances~\cite{DD16} contains a large number of distances including metrics on such objects as graphs, matrices, strings, permutations, etc. These distances are not only interesting to be studied from a theoretical viewpoint, but have also importance in applications.

Many real life settings require comparisons between pairs of objects that cannot be described simply by the Euclidean metric. As an example, the question how to compare different rankings arose in many psychological works, where two or more different observers had the task to rank a set of objects for certain properties. The search for a quantification of the similarity of rankings led to the famous Kendall's rank correlation coefficient introduced in~\cite{K38}. The Kendall $\tau$ distance arises naturally from the rank correlation coefficient and counts the number of pairwise disagreements between two permutations. It is an example of a metric defined on the set $S_n$ of all permutations or, equivalently, ranking lists of order $n$, see~(\ref{e3}). A historical review of Kendall's $\tau$ and related coefficients can be found in~\cite{K58}.

The first known to authors weighted generalization of the Kendall $\tau$ metric was introduced in~\cite{LH06,LZH08} as follows:
\begin{equation}\label{e1}
K_w(\pi,\varphi)=\sum\limits_{1\leqslant i<j\leqslant n}w_iw_j\mathbf{1}_{(\pi_j-\pi_i)(\varphi_j-\varphi_i)<0},
\end{equation}
where $w_iw_j>0$, and $\pi,\varphi \in S_n$.

This generalization was considered in order to ``accommodate'' the fact that not all predicates are equally important. The metric $K_w$ is equal to the standard Kendall $\tau$ distance if $w_i=1$ for all $i\in \{1,..,n\}$. Without any reference to~\cite{LH06}, apparently independently, similar generalizations appeared in~\cite{KV10}  and~\cite{LY10,LY12}. In the latter case authors were motivated from the weighted Kendall's $\tau$ correlation coefficient proposed in~\cite{S98}. In the literature there exist diverse generalizations and modifications of Kendall's $\tau$ distance. It is worth to mention a generalization on partial orders~\cite{BGH13}, generalizations in terms of transpositions~\cite{CV14,BE14}, the so-called probabilistic Kendall's tau distance~\cite[p.~125]{VBNK17}, a weighted Kendall's distance \cite{FTM12}, and Kendall's tau sequence distance~\cite{C19}.

Aside from modifications of Kendall's $\tau$, several other noteworthy distance and similarity measures for permutations, respectively rankings, emerged from applications. Another well-known proximity measure for rankings is Spearman's $\varrho$ rank coefficient, introduced in \cite{S04}. More recent specialized approaches can be found, e.g., in the application fields of medical diagnostics~\cite{KW04}, hydrology \cite{FSS17}, physiology \cite{WSSC16}, and neurophysiology \cite{ODRL2010}. An overview over metrics on $S_n$ is given in \cite{DH98}.

In this paper, using a matrix of weights $W$, we consider a more general than~(\ref{e1}), but for convenience normalized, distance $d_W$: instead of multiplications of weights $w_iw_j$ we use the weights $w_{i,j}$ directly assigned to the pair of coordinates $(i,j)$, see~(\ref{e11}), and study  geometric properties of the space $(S_n,d_W)$. Such generalization is interesting from the theoretical point of view and is more flexible for possible applications.

In Section 2 we prove that $(S_n,d_W)$ in general case is a pseudometric space and is a metric space if and only if all weights are positive, which is stated in Theorem~\ref{thm:metric}. Furthermore, we describe several properties of this space and its metric.

The main result of Section 3 is presented in  Theorem~\ref{th3.4}, where we provide a criterion when a point from $(S_n,d_W)$ ``lies between'' another two fixed points from $(S_n,d_W)$. It is formulated using an edge-graph of a permutohedron of order $n$.

In Section 4 we investigate the occurrence of special type four-point subsets of the space $(S_n,d_W)$ -- so called  ``pseudolinear quadruples''.

At the end of the paper we formulate a conjecture that suggests a characterization of the metric space $(S_n,d_W)$ by certain geometric properties in a sense that all metric spaces $(X,d)$, $|X|=n!$, with these properties are isometric to $(S_n,d_W)$ for some weight $W$. We prove it for the case $n=3$.

\section{Definitions and basic properties}

In this section we introduce a weighted generalization of Kendall's tau distance and study some of its basic properties.

Denote by $S_n$ the set of all permutations of the numbers $1,...,n$.
For the permutations $\pi=(\pi_1,\dots,\pi_n)$, $\varphi=(\varphi_1,\dots,\varphi_n)\in S_n$ define the {discordance indicator} by
$$
\dsc_{i,j}(\pi,\varphi)=\mathbf{1}_{(\pi_j-\pi_i)(\varphi_j-\varphi_i)<0}.
$$
Clearly, $\dsc_{i,j}(\pi,\varphi)$ is symmetric with respect to $i, j$ and with respect to $\pi, \varphi$. Define the discordance set of $\pi$ and $\varphi$ to be
$$
\dsc(\pi,\varphi)=\{(i,j) \, | \, i<j, \, \dsc_{i,j}(\pi,\varphi)=1\}.
$$

\begin{remark}\label{rem1}
In the special case where $\varphi$ is the identity permutation $id$, $\dsc(\pi,id)$ describes the set of inversions, i.e., all index pairs $i<j$ with $\pi_i>\pi_j$. Typically, the notation $I(\pi)$ is chosen for the inversion set, the cardinality $\left\vert I(\pi)\right\vert$ of the inversion set is called the inversion number of the permutation $\pi$ and is a well known property of permutations that measures the sortedness. It is related to the sign of a permutation and was first introduced and used in the Cramer's rule for determinants.
\end{remark}

Let $W=(w_{i,j})\in\RR_{+}^{n\times n}$ be a {strictly upper triangular} weighting matrix with $w_{i,j}\geqslant 0$.
Define a map $d_W\colon S_n\times S_n \to [0,1]$ as follows:
\begin{equation}\label{e11}
d_W(\pi,\varphi)=\frac{\sum_{i,j=1}^n \dsc_{i,j}(\pi,\varphi)\cdot w_{i,j}}{\sum_{i,j=1}^n w_{i,j}}.
\end{equation}
Everywhere below we consider that $\sum_{i,j=1}^n w_{i,j}\neq 0$.
The distance~(\ref{e11}) generalizes the normalized Kendall $\tau$ ranking distance, which is defined as
\begin{equation}\label{e3}
K(\pi,\varphi) = \frac{\vert\dsc(\pi,\varphi)\vert}{n(n-1)/2}.
\end{equation}
This distance coincidences with $d_W(\pi,\varphi)$ if $W=W^{\tau}=(w_{i,j})$, where $w_{i,j}=1$, $i<j$. Thereby, $d_{W^{\tau}}$ is also related to Kendall's $\tau$ correlation coefficient $\tau(\pi,\varphi)$~\cite{K38} by $\tau(\pi,\varphi)=1-2 d_{W^{\tau}}$.

Recall that a {pseudometric space} is a generalization of a metric space in which the distance between two distinct points can be zero, i.e., instead of axiom (ii) in the definition of metric spaces we have the condition $d(x,x)=0$. In this case $d$ is called a {pseudometric}.

\begin{theorem}\label{thm:metric}
The pair $(S_n,d_W)$ is a pseudometric space, $n\geqslant 2$. Moreover, $d_W$ is a metric if and only if all the weights $w_{i,j}$ are positive for $i<j$.
\end{theorem}
\begin{proof}
Symmetry can be seen immediately. Zero distance $d_W(\pi,\pi)=0$ for equal permutations follows from the equality $\dsc_{i,j}(\pi,\pi)=\mathbf{1}_{(\pi_j-\pi_i)^2<0}=0$. Let $\pi, \varphi, \psi \in S_n$. To prove the triangle inequality consider the difference
\begin{equation*}
\Delta=d_W(\pi,\varphi)+d_W(\varphi,\psi) - d_W(\pi,\psi).
\end{equation*}
In order to fulfill the triangle inequality, this difference has to be nonnegative. One can see that $\Delta$ can be written as

\begin{equation}\label{eq1}
\Delta=\frac{\sum_{i,j=1}^n \left(\dsc_{i,j}(\pi,\varphi) +\dsc_{i,j}(\varphi,\psi)-\dsc_{i,j}(\pi,\psi)\right)w_{i,j}}{\sum_{i,j=1}^n w_{i,j}} \geq 0.
\end{equation}
The summands $\dsc_{i,j}(\pi,\varphi)+\dsc_{i,j}(\varphi,\psi)-\dsc_{i,j}(\pi,\psi)$ in the numerator only takes values in $\lbrace 0,2\rbrace$ since
\begin{align*}
(\dsc_{i,j}(\pi,\varphi)=\dsc_{i,j}(\varphi,\psi)) &\Rightarrow (\dsc_{i,j}(\pi,\psi)=0), \\
(\dsc_{i,j}(\pi,\varphi)\neq\dsc_{i,j}(\varphi,\psi)) &\Rightarrow (\dsc_{i,j}(\pi,\psi)=1).
\end{align*}
Hence and because of the fact that all weights $w_{i,j}$ are nonnegative, inequality~(\ref{eq1}) evidently holds. Thereby the triangle inequality holds for $d_W$ making it a pseudometric.

Consider now only positive weights $w_{i,j}$ for $i<j$. Then
\begin{equation*}
d_W(\pi,\varphi)=0 \Leftrightarrow \dsc_{i,j}(\pi,\varphi)=0 \, \text{ for all } \, 1\leqslant i<j \leqslant n.
\end{equation*}
Since both permutations are thereby concordant for every index pair $i,j$, they must be equal and $d_W$ satisfies the identity of indiscernibles.
\end{proof}
\begin{remark}
The requirement of positive weights in~(\ref{e11}) is sharp in the sense that if $w_{i,j}=0$ for at least one pair $(i,j)$ one can always find pairs of permutations with distance zero. These can be obtained by swapping the $i$th and $j$th element of an arbitrary permutation.
\end{remark}

\begin{remark}
In the case $n=1$ we have $S_1=\{(1)\}$ and $(S_1,d_W)$ is a trivial one-point metric space.
\end{remark}

Let us define the following permutation for given $\pi=(\pi_1,\dots,\pi_n)$:
\begin{align*}
\hat{\pi}&=(n+1-\pi_1,\dots,n+1-\pi_n) \text{\quad({ordinal inverse})}.
\end{align*}

The next proposition describes some basic geometric properties of the space $(S_n,d_W)$.

\begin{proposition}\label{p22}
The following conditions hold for every pseudometric space $(S_n,d_W)$, $n\geqslant 2$, and for all $\pi, \varphi \in S_n$:
\begin{itemize}
\item[\emph{(i)}] The pseudometric $d_W$ is scaling invariant with respect to $W$, i.e., the equality  $$d_{aW}(\pi,\varphi)=d_{W}(\pi,\varphi)$$ holds  for every  $a>0$, and for every $W\in\RR^{n\times n}_{+}$.

\item[\emph{(ii)}] The pseudometric  $d_W$ is subadditive with respect to $W$, i.e., the inequality
$$d_{W+V}(\pi,\varphi)\leqslant d_{W}(\pi,\varphi)+d_{V}(\pi,\varphi)$$
holds for all  $W, V\in\RR^{n\times n}_+$.

\item[\emph{(iii)}] The equality $d_W(\pi,\hat{\varphi})=1-d_W(\pi,\varphi)$ holds.

\item[\emph{(iv)}] The equality $d_W(\pi,\varphi)=d_W(\hat{\pi},\hat{\varphi})$ holds.

\item[\emph{(v)}] The equality $d_W(\pi,\varphi)\leqslant 1$ holds. Moreover, if $d_W$ is a metric, then $d_W(\pi,\varphi)=1$ if and only if $\varphi=\hat{\pi}$.
\end{itemize}
\end{proposition}

\begin{proof}
The proofs of the statements (i) and (ii) are straightforward and left to the reader.

(iii) For each index pair $i<j$ it follows from the definition of the ordinal inverse that
\begin{equation*} \sgn(\hat{\varphi}_j-\hat{\varphi}_i)=\sgn(n+1-\varphi_j-n-1+\varphi_i)=\sgn(\varphi_i-\varphi_j)=-\sgn(\varphi_j-\varphi_i).
\end{equation*}
In consequence, $\dsc_{i,j}(\pi,\hat{\varphi})=\mathbf{1}_{(\pi_j-\pi_i)(\varphi_j-{\varphi}_i)>0}=1-\dsc_{i,j}(\pi,\varphi)$.
Since this statement holds for every index pair $i<j$, the result remains when the weighted and normalized sum over all pairs is considered, leading to (iii).

(iv) Using statement (iii) we have $d_W(\hat{\pi},\hat{\varphi})=1-d_W(\hat{\pi},\varphi)=1-d_W(\varphi,\hat{\pi})=1-(1-d_W(\varphi,\pi))=d_W(\pi,\varphi)$.

(v) The inequality $d_W(\pi,\varphi)\leqslant 1$ follows directly from~(\ref{e11}). Let $d_W$ be a metric and let $\varphi=\hat{\pi}$. Using conditions (iv) and (iii) for the proof consider the sequence of equivalences: $\varphi=\hat{\pi}$ iff $d_W(\hat{\pi},\varphi)=0$ iff $d_W(\pi,\hat{\varphi})=0$ iff $1-d_W(\pi,\varphi)=0$ iff  $d_W(\pi,\varphi)=1$.
\end{proof}

\section{Betweenness of points in $(S_n,d_W)$}

In this section we continue to study geometric properties of the space $(S_n,d_W)$ by characterizing triplets of points from $S_n$ which satisfy the ternary relation ``to lie between''. This relation is intuitive for points belonging to some straight line, plane or three-dimensional space. K. Menger~\cite[p.~77]{Me28} seems to be the first who formulated the concept of ``metric betweenness'' for general metric spaces. Let $(X, d)$ be a metric space, and let $x, y$ and $z$ be different points from $X$. The point $y$ {lies between} $x$ and $z$, if $d(x, z) = d(x, y) + d(y, z)$. This concept is used also in the present time for the study of metric spaces, see, e.g.,~\cite{ACH16}.

Recall that an {undirected graph} is a pair $(V,E)$ consisting of a nonempty set $V$ and a (probably empty) set $E$ whose elements are unordered pairs of different points from $V$. For a graph $G=(V,E)$, the sets $V=V(G)$ and $E=E(G)$ are called {the set of vertices} and {the set of edges}, respectively. A {path} in a graph $G$ is a subgraph $P$ of $G$ for which
$$
V(P)=\{x_0,...,x_k\}, \quad E(P) =\{\{x_0,x_1\},...,\{x_{k-1},x_k\}\},
$$
where all $x_i$ are distinct. Sometimes for convenience we refer to a path by the natural sequence of its vertices, say, $P=\{x_0,...,x_k\}$.
A finite graph $C$ is a cycle if $|V(C)| \geq 3$ and there exists an enumeration $(v_1, \ldots, v_n)$ of its vertices such that
$$
(\{v_i, v_j\} \in E(C)) \Leftrightarrow (|i-j|=1 \text{ or } |i-j|=n-1).
$$
A  cycle is {simple} if no repetitions of vertices and edges allowed.

\begin{figure}[ht]
\includegraphics[width=\textwidth]{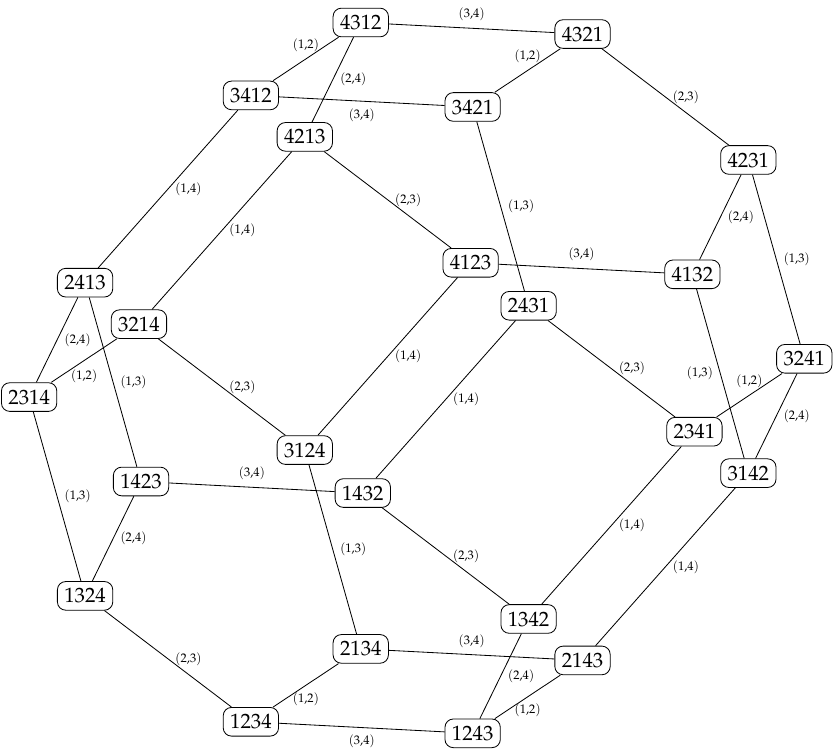}
\caption{The labeled undirected graph $G_4$.}
\label{f1}
\end{figure}

Denote by $\tr_{ij}(\pi)$ a permutation which is obtained from $\pi=\{\pi_1,...,\pi_n\}$ by transposition of two elements $\pi_i$ and $\pi_j$, $i\neq j$, $i,j\in \{1,...,n\}$, only in the case if
\begin{equation}\label{e12}
 |\pi_i-\pi_j|=1.
\end{equation}

Let $G_n=G_n(V,E)$ be an undirected graph such that $V(G_n)=S_n$ and $\{\pi,\varphi\}\in E(G_n)$ if and only if $\varphi=\tr_{ij}(\pi)$ for some $i\neq j$. The graph $G_n$ is known as an edge-graph of a permutohedron of order $n$. Let us remember that an adjacent transposition is a transposition $(\pi_i \, \, \pi_j)$ where the two elements are consecutive, i.e., when equality~(\ref{e12}) holds. In other words, two vertices of the graph $G_n$ are connected by the edge if and only if one vertex is obtained from the other by applying an adjacent transposition. For the edges $E(G_n)$ we define a labeling function $l\colon E(G_n)\to \{(i,j) \, | \, i<j\}$ in the following way:
let $e$ be the edge $\{\pi,\varphi\}$, then it is labeled with $l(e)=(i,j)$, which is the index pair for which $\varphi=\tr_{ij}(\pi)$ holds. The labeled graph $G_4$ is depicted at Figure~\ref{f1}.

Recall that a permutohedron of order $n$ is an $(n-1)$-dimensional polytope embedded in an $n$-dimensional Euclidean space, which is a convex hull of all $n!$ points formed by permuting the coordinates of the vector $(1,2,\ldots,n)$. Furthermore, the permutohedron is vertex-transitive, i.e., for every $\tau\in S_n$ the following implication holds:
$$
\{\pi,\varphi\}\in E(G_n) \Rightarrow \{\pi\circ\tau,\varphi\circ\tau\}\in E(G_n).
$$
This is indeed true, since $\varphi=\tr_{ij}(\pi)$ implies $\varphi\circ\tau = \tr_{\tau_i,\tau_j}(\pi\circ\tau)$, as $\tau$ only reorders the indices. In consequence, permutohedra are highly symmetric, which can be seen in the visualizations of $G_n$ for the cases $n=4$ and $n=3$. Both are depicted in Figures~\ref{f1} and~\ref{f2}, respectively.

%From now on, by the length of a path we understand the number of its edges.
%by $\mathfrak P_{\pi,\varphi}$  we denote the set of all these shortest paths between $\pi$ and $\varphi$.The length of the shortest paths is closely related to the discordance set by its cardinality, as the following proposition shows.
Recall that a distance $d_G$ between two vertices in a connected graph $G$ is the number of edges in a shortest path connecting them. The distance $d_G$ satisfies axioms (i)--(iii) of a metric space. Thus, $d_G$ is metric induced by $G$. In general, the shortest path connecting two any different vertices of $G_n$ is not obligatory unique. For the cases $n=1$ and $n=2$, the graphs are given by $V(G_1)=\{(1)\}$, $E(G_1)=\varnothing$ and $V(G_2)=\{(1,2), (2,1)\}$, $E(G_2)=\{\{(1,2), (2,1)\}\}$ respectively.

\begin{proposition} \label{p01}
For any $n\geqslant 1$ and for any $\pi, \varphi \in S_n$ the distance $d_{G_n}$ between these permutations equals the number of discordant pairs $|\dsc(\pi,\varphi)|$.
\end{proposition}
%\begin{proposition} \label{p01}
%Let $\pi, \varphi$ be different vertices of $G_n$ and $\mathcal P$ be a path in $G_n$ connecting $\pi$ and $\varphi$. Then $\mathcal P\in \mathfrak P_{\pi,\varphi}$ if and only if $|E(\mathcal P)|=|\dsc(\pi,\varphi)|$.
%\end{proposition}

\begin{proof}
For $n=1,2$ the statement is trivial. Let $n\geqslant 3$. An adjacent transposition on a discordant pair decreases the inversion number by exactly one. Indeed, in this case only two consecutive integers are swapped and all other elements of permutation preserve their order in relation to each of the swapped elements. So every path between $\pi$ and $\varphi$ require at least $|\dsc(\pi,\varphi)|$ edges.

Let us show the existence of such path by induction on $|\dsc(\pi,\varphi)|$. Since the graph is vertex transitive we may assume that $\pi = id$. If $|\dsc(id,\varphi)|=0$, then $\varphi=id$ and $d_{G_n}(id,id)=0$. Let $\varphi = (\varphi_1,...,\varphi_n)$ and let $|\dsc(id,\varphi)|\neq 0$.
%It is clear that in this case $\varphi \neq id$.
Hence there exists at least one pair of elements $\varphi_i$ and $\varphi_j$ such that $i<j$ and $\varphi_i-\varphi_j=1$. Indeed, assuming the opposite we immediately get that 1 must be before 2, 2 before 3,... and, in consequence, $\varphi=id$. Let $\hat{\varphi}=\tr_{ij}(\varphi)$.
Then $|\dsc(id,\hat{\varphi})|=|\dsc(id,\varphi)|-1$. By induction hypothesis $d_{G_n}(id,\hat{\varphi})=|\dsc(id,\hat{\varphi})|$. Since $\dsc(id,{\varphi})=\dsc(id,\hat{\varphi})\cup\{(i,j)\}$ we obtain the necessary equality. %\mbox{$d_{G_n}(u,v)=|\dsc(\pi,\varphi)|$}.
\end{proof}

Let $\mathcal P= \{\pi=\psi_1,\ldots,\psi_{n}=\varphi\}$ be a path joining $\pi$ and $\varphi$ in $G_n$. By the definition of $G_n$ we have
\begin{equation}\label{e13}
\begin{split}
\psi_2=\tr_{i_1,j_1}(\psi_1) \text{ and } &\dsc(\psi_1,\psi_2)=\{(i_1,j_1)\}, \\
&\cdots \\
\psi_n=\tr_{i_{n-1},j_{n-1}}(\psi_{n-1}) \text{ and } &\dsc(\psi_{n-1},\psi_n)=\{(i_{n-1},j_{n-1})\}.
\end{split}
\end{equation}

The following corollary follows directly from Proposition~\ref{p01} and the definition of the graph $G_n$.
\begin{corollary}\label{c12}
Let $\pi, \varphi$ be different vertices of $G_n$ and let $\mathcal P$ be a path in $G_n$ connecting $\pi$ and $\varphi$. Then $\mcp$ is a shortest-path if and only if all the pairs $(i_1, j_1),\ldots,(i_{n-1}, j_{n-1})$ defined by~(\ref{e13}) are different.
\end{corollary}

\begin{proposition}\label{p13}
Let $\pi,\varphi$ be different vertices of the graph $G_n$, $n\geqslant 2$. Then the following statements are equivalent for every $\psi\in V(G_n)$:
\begin{itemize}
\item [\emph{(i)}] $\dsc(\pi,\varphi)=\dsc(\pi,\psi)\cup\dsc(\psi,\varphi)$.
\item [\emph{(ii)}] $\dsc(\pi,\psi)\cap\dsc(\psi,\varphi) = \varnothing$.
\item [\emph{(iii)}] There exists a shortest-path $\mathcal{P}_{\pi,\varphi}$ between $\pi$ and $\varphi$ such that $\psi\in V(\mathcal{P}_{\pi,\varphi})$.
\end{itemize}
\end{proposition}
\begin{proof}
Let us prove the implication (i)$\Rightarrow$(ii) by contradiction. Let us assume that there exists a pair
\begin{equation}\label{e311}
(i,j)\in\dsc(\pi,\psi)\cap\dsc(\psi,\varphi).
\end{equation}
Since $\dsc(\pi,\psi)\cap\dsc(\psi,\varphi)\subseteq \dsc(\pi,\psi)\cup\dsc(\psi,\varphi)$, by (i) we have $(i,j)\in\dsc(\pi,\varphi)$.  From~(\ref{e311}) it follows that $\dsc_{i,j}(\pi,\psi)=\dsc_{i,j}(\psi,\varphi)=1$ or equally
 \begin{equation*}
 (\pi_j-\pi_i)(\psi_j-\psi_i)<0  \ \text{ and } \  (\psi_j-\psi_i)(\varphi_j-\varphi_i)<0.
 \end{equation*} Multiplying both left sides gives
\begin{align*}
&\quad(\pi_j-\pi_i)(\psi_j-\psi_i)^2(\varphi_j-\varphi_i) >0 \\
\Rightarrow&\quad (\pi_j-\pi_i)(\varphi_j-\varphi_i) >0 \\
\Rightarrow&\quad \dsc_{i,j}(\pi,\varphi) =0 \\
\Rightarrow&\quad (i,j) \notin \dsc(\pi,\varphi).
\end{align*}
This contradicts to our assumption.

(ii)$\Rightarrow$(iii) Let $\mathcal{P}_{\pi,\psi}$  and $\mathcal{P}_{\psi,\varphi}$ be any shortest paths between the respective vertices and let $\mathcal{P}_{\pi,\varphi}=\mathcal{P}_{\pi,\psi} \cup \mathcal{P}_{\psi,\varphi}$ be the compound path from $\pi$ to $\varphi$ over $\psi$. By (ii) all the pairs $(i_1, j_1),\ldots,(i_{n-1}, j_{n-1})$ defined by~(\ref{e13}) for the path $\mathcal{P}_{\pi,\varphi}$ are different. Therefore by Corollary~\ref{c12} the compound $\mathcal{P}_{\pi,\varphi}$ is a shortest path between $\pi$ and $\varphi$.

(iii) $\Rightarrow$ (i)
%Let there exist $\mathcal P_{\pi,\varphi}$ such that $\psi\in V(\mathcal{P}_{\pi,\varphi})$ and
Let the relations~(\ref{e13}) hold for $\mcp_{\pi,\varphi}$. By Corollary~\ref{c12} all the pairs $(i_1,j_1),...,(i_{n-1},j_{n-1})$ are different. From~(\ref{e12}) it follows that
$$
\dsc(\psi_1,\psi_{k+1})=\dsc(\psi_1,\psi_k)\cup \{(i_k,j_k)\},\quad k=1,...,n-1.
$$
Hence,
$$
\dsc(\pi,\psi_k)=\dsc(\psi_1,\psi_k)=\{(i_1,j_1),...,(i_{k-1},j_{k-1})\}, \quad k=2,...,n-1,
$$
$$
\dsc(\pi,\varphi)=\dsc(\psi_1,\psi_n)=\{(i_1,j_1),...,(i_{n-1},j_{n-1})\}.
$$
Analogously,
$$
\dsc(\psi_k,\varphi)=\{(i_k,j_k),...,(i_{n-1},j_{n-1})\} \quad k=2,...,n-1,
$$
which establishes~(i) with $\psi=\psi_k$.
\end{proof}

%A metric defined over a set of vertices $V(G)$ in terms of distances in a graph is called a graph metric or metric induced by $G$.

\begin{theorem}\label{th3.4}
For any $n\geqslant 3$ and any weighting matrix $W$ (strictly upper triangular and positive), and any three different permutations $\pi$, $\psi$, $\varphi \in S_n$, $\psi$ lies between $\pi$ and $\varphi$  with respect to the metric $d_{G_n}$ if and only if $\psi$ lies between $\pi$ and $\varphi$ with respect to $d_W$.

%Let $\pi,\psi,\varphi$ be different elements of $S_n$, $n\geqslant 3$. Then the following conditions are equivalent:
%\begin{itemize}
%  \item [\emph{(i)}] $\psi$ is a point belonging to one of the the shortest-paths  connecting $\pi$ and $\varphi$ in $G_n$.

% \item [\emph{(ii)}] $\psi$ lies between $\pi$ and $\varphi$ in the metric space $(S_n,d_W)$, i.e., the equality
%\begin{equation}\label{e35}
%d_W(\pi,\varphi)=d_W(\pi,\psi)+d_W(\psi,\varphi)
%\end{equation}
%holds.
%\end{itemize}
\end{theorem}

\begin{proof}
Let $\psi$ lie between $\pi$ and $\varphi$ with respect to $d_{G_n}$. Then, $\psi$ belongs to some shortest path $\mathcal P_{\pi, \varphi}$ connecting $\pi$ and $\varphi$ in $G_n$. The distance $d_W$ is the normalized sum of weights associated to discordant pairs between two permutations. These discordant pairs are exactly the labels of edges on shortest paths between those two permutations.
By Corollary~\ref{c12}  the labels on the shortest-path $\mathcal P_{\pi, \varphi}$ are exactly the disjoint union of the labels on the shortest-paths $\mathcal P_{\pi, \psi}$ and $\mathcal P_{\psi, \varphi}$. Hence we have the equality $d_W(\pi,\varphi)=d_W(\pi,\psi)+d_W(\psi,\varphi)$.

%The implication (i)$\Rightarrow$(ii) follows directly from the implication (iii)$\Rightarrow$(i) of Proposition~\ref{p13} and~(\ref{e11}).

Let us show the converse implication by contradiction.
Let $\psi$ be a permutation that does not lie between $\pi$ and $\varphi$ with respect to $d_{G_n}$.
Then $\psi$ does not belong to any shortest path connecting $\pi$ and $\varphi$ in $G_n$. Proposition~\ref{p13} implies
$\dsc(\pi,\varphi)\neq\dsc(\pi,\psi)\cup\dsc(\psi,\varphi)$.
Now if $\dsc_{i,j}(\pi,\varphi)=1$, there are two possibilities:
$$\dsc_{i,j}(\pi,\psi)=1 \ \text{ and }\ \dsc_{i,j}(\psi,\varphi)=0$$
or
$$\dsc_{i,j}(\pi,\psi)=0  \ \text{ and }\ \dsc_{i,j}(\psi,\varphi)=1.$$
In other words $(i,j) \in \dsc(\pi,\varphi)$ implies $(i,j) \in \dsc(\pi,\psi)\cup\dsc(\psi,\varphi)$. Hence, $\dsc(\pi,\varphi)$ is a proper subset of $\dsc(\pi,\psi)\cup\dsc(\psi,\varphi)$ and the equality
\begin{equation*}%\label{e36}
\frac{\sum_{i,j=1}^n \dsc_{i,j}(\pi,\varphi)\cdot w_{i,j}}{\sum_{i,j=1}^n w_{i,j}} = \frac{\sum_{i,j=1}^n \dsc_{i,j}(\pi,\psi)\cdot w_{i,j}}{\sum_{i,j=1}^n w_{i,j}}+\frac{\sum_{i,j=1}^n \dsc_{i,j}(\psi,\varphi)\cdot w_{i,j}}{\sum_{i,j=1}^n w_{i,j}}
\end{equation*}
is impossible, since $d_W$ is a metric and all $w_{i,j}$ are positive.
\end{proof}

\section{Pseudolinear quadruples in $(S_n,d_W)$}

The aim of this section is to describe the occurrence of special type four-point subsets of the space $(S_n,d_W)$, the so called ``pseudolinear quadruples''.

In 1928 K. Menger~\cite{Me28} proved that if every three points of a metric space $X$, $|X|\geqslant 3$, are embeddable into $\RR^1$, then $X$ is isometric to some subset of $\RR^1$ or $X$ is a pseudolinear quadruple. Recall that a four-point metric space $(X,d)$ is called a {pseudolinear quadruple} if there exists an enumeration $x_1, x_2, x_3, x_4$ of the points of $X$ such that the equalities
\begin{gather}\label{e21}
d(x_1,x_2)=d(x_3,x_4)=s, \, \, d(x_2,x_3)=d(x_4,x_1)=t, \\
d(x_2,x_4)=d(x_3,x_1)=s+t \notag
\end{gather}
hold with some positive reals $s$ and $t$. Note also that equilateral pseudolinear quadruples are known by their extremal properties~\cite{DP11}.

Let $(X,d)$ be a metric space. Recall that for every nonempty set $A\subseteq X$ the quantity
$$\diam A = \sup \{d(x,y)\colon x,y \in A\}$$
is the {diameter} of $A$. We shall say that points $a, b$ are {diametrical} for the set $A$ if $d(a,b)=\diam A$.

Everywhere below in this section we consider that $n\geqslant 3$, since this is a necessary condition for the existence of pseudolinear quadruples in $(S_n,d_W)$.

\begin{proposition}\label{p14}
Let $\pi$ and $\varphi$ be nondiametrical points in the pseudometric space $(S_n,d_W)$. Then the set $X=\{\pi,\varphi,\hat{\pi},\hat{\varphi}\}$  forms a pseudolinear quadruple.
\end{proposition}
\begin{proof}
By condition (iv) of Proposition~\ref{p22} we have
$$
d_W(\pi,\varphi)=d_W(\hat{\pi},\hat{\varphi}),\quad
d_W(\pi,\hat{\varphi})=d_W(\varphi,\hat{\pi}).
$$
Using conditions (iii) and (v) of the same proposition  we obtain the following equalities:
$$
d_W(\pi,\varphi)+d_W(\varphi,\hat{\pi})
=d_W(\pi,\varphi)+1-d_W(\varphi,\pi)=1=d_W(\pi,\hat{\pi}),
$$
$$
d_W(\varphi,\pi)+d_W(\pi,\hat{\varphi})
=d_W(\varphi,\pi)+1-d_W(\pi,\varphi)=1=d_W(\varphi,\hat{\varphi}).
$$
Thus, $(X,d_W)$ is a pseudolinear quadruple with $x_1=\pi$, $x_2=\varphi$, $x_3=\hat{\pi}$, $x_4=\hat{\varphi}$.
\end{proof}

%%copied
It is easy to see that every cycle in $G_n$ is even.
Let $C$ be a labeled cycle in $G_n$.
We shall say that $C$ has a symmetric labeling if $l(e)=l(\bar{e})$, where $\bar{e}$ is an edge opposite to $e$ in $C$.
Denote by $E_C(i,j)$ the set of edges of a cycle $C$ labeled by the label $(i,j)$.

\begin{proposition}\label{p15}
Let $C$ be a simple cycle in $G_n$ having a symmetric labeling.
Let for every label $(i,j)$ of the cycle $C$ the equality $|E_C(i,j)|=2(2k-1)$ hold for some $k\in \NN^+$. Then
for every different non opposite vertices   $\pi,\varphi$ of $C$  the set $\{\pi,\varphi,\bar{\pi},\bar{\varphi}\}$ form a pseudolinear quadruple in $(S_n,d_W)$,
where $\bar{\pi},\bar{\varphi}$ are opposite vertices to $\pi,\varphi$ in $C$, respectively.
\end{proposition}

\begin{proof}
Let $\pi\in V(C)$ and let $\mcp_{\pi,\bar{\pi}}$ be one of the paths connecting $\pi$ and $\bar{\pi}$ in $C$. Without loss of generality, consider that $\varphi \in V(\mcp_{\pi,\bar{\pi}})$. Since $C$ has a symmetric labeling, for every label $(i,j)$ of $C$ the number of edges labeled by $(i,j)$ and belonging to $\mcp_{\pi,\bar{\pi}}$ is odd. Hence, the number of edges labeled by $(i,j)$ and belonging to $\mcp_{\pi,\varphi}$ ($\mcp_{\varphi,\bar{\pi}}$)  is odd (even) or vice versa. Thus, $\dsc_{i,j}(\pi,\varphi)=1$ ($\dsc_{i,j}(\varphi,\bar{\pi})=0$) or vice versa and $\dsc_{i,j}(\pi,\bar{\pi})=1$. Anyway,
\begin{equation}\label{e22}
 \dsc_{i,j}(\pi,\varphi)+\dsc_{i,j}(\varphi,\bar{\pi})=\dsc_{i,j}(\pi,\bar{\pi})
\end{equation}
 for every label $(i,j)$ of the graph $C$. Analogously,
\begin{equation}\label{e23}
 \dsc_{i,j}(\pi,\varphi)+\dsc_{i,j}(\pi,\bar{\varphi})=\dsc_{i,j}(\varphi,\bar{\varphi}).
\end{equation}
 Equalities~(\ref{e22}),~(\ref{e23}) and~(\ref{e11}) give
$$
d_W(\pi,\varphi)+ d_W(\varphi,\bar{\pi})= d_W(\pi,\bar{\pi}),
$$
$$
d_W(\pi,{\varphi})+ d_W(\pi,\bar{\varphi})= d_W(\varphi,\bar{\varphi}).
$$
By symmetric labeling of $C$ we have
$$
 \dsc(\pi,\varphi)=\dsc(\bar{\pi},\bar{\varphi}),\quad  \dsc(\varphi,\bar{\pi})=\dsc(\pi,\bar{\varphi}).
$$
 Hence, by~(\ref{e11})
 $$
 d_W(\pi,\varphi)=d_W(\bar{\pi},\bar{\varphi}),\quad  d_W(\varphi,\bar{\pi})=d_W(\pi,\bar{\varphi}).
 $$
 Thus, equalities~(\ref{e21}) are satisfied with
 $$
 x_1=\pi, \, x_2=\varphi, \, x_3=\bar{\pi}, \, x_4=\bar{\varphi}, \text{ and }
 $$
 $$
 s=d_W(\pi,\varphi)=d_W(\bar{\pi},\bar{\varphi}),\quad  t=d_W(\varphi,\bar{\pi})=d_W(\pi,\bar{\varphi}).
 $$
\end{proof}

In the case $k=1$ Proposition~\ref{p15} implies the following.
\begin{corollary}\label{c16}
Let $C$ be a simple cycle in $G_n$ having a symmetric labeling and let for every $\pi \in V(C)$ different edges of the path $\mcp_{\pi,\bar{\pi}}\subseteq C$ have different labels, where $\bar{\pi}$ is a vertex opposite to $\pi$ in $C$. Then for every different nonopposite vertices   $\pi,\varphi$ of $C$  the set $\{\pi,\varphi,\bar{\pi},\bar{\varphi}\}$ form a pseudolinear quadruple in $(S_n,d_W)$, where $\bar{\pi},\bar{\varphi}$ are opposite vertices to $\pi,\varphi$ in $C$, respectively.
\end{corollary}

\begin{remark}
The assertion converse to Corollary~\ref{c16} does not hold. Consider the permutations

$$\alpha=(1,2,3,4), \quad \beta=(4,1,2,3), \quad \gamma=(4,2,3,1), \quad \delta=(1,3,4,2).$$
We have
\begin{align*}
\dsc(\alpha,\beta)&=\{(1,2),(1,3),(1,4)\}=\dsc(\gamma,\delta),\\
\dsc(\beta,\gamma)&=\{(2,4),(3,4)\}=\dsc(\alpha,\delta),\\
\dsc(\alpha,\gamma)&=\{(1,2),(1,3),(1,4),(2,4),(3,4)\}=\dsc(\beta,\delta).
\end{align*}
This implies that $(\alpha,\beta,\gamma,\delta)$ is a pseudolinear quadruple in $(S_4,d_W)$. 	Let us show that this pseudolinear quadruple is not a part of a symmetric labeled cycle in $G_4$. Therefore, we show that there are no paths $P_{\alpha,\beta}$ from $\alpha$ to $\beta$ and $P_{\gamma,\delta}$ from $\gamma$ to $\delta$ in $G_4$ such that they have the same length and the same labeling. Denote by $l_a(\pi)$ the set of all labels of the edges adjacent to $\pi$. One can see from Figure~\ref{f1} that
\begin{align*}
l_a(\alpha)&=\{(1,2),(2,3),(3,4)\},\\
l_a(\gamma)&=\{(2,4),(2,3),(1,3)\}.
\end{align*}
Hence, $l_a(\alpha)\cap l_a(\gamma)=\{(2,3)\}$.
For the next point $\alpha_1$ on $P_{\alpha,\beta}$ and the next point $\gamma_1$ on $P_{\gamma,\delta}$ the labels must be equal, therefore $l(\{\alpha,\alpha_1\})=l(\{\gamma,\gamma_1\})=(\{2,3\})$. Consequently, $\alpha_1=(1,3,2,4)$, $\gamma_1=(4,3,2,1)$. Again
\begin{align*}
l_a(\alpha_1)&=\{(2,3)(1,4),(2,4)\},\\
l_a(\gamma_1)&=\{(2,3)(1,2),(3,4)\},
\end{align*}
and $l_a(\alpha_1)\cap l_a(\gamma_1)=\{(2,3)\}$.
Thus, there is no other way than backwards for the labels to be symmetric. In conclusion, there are no symmetrically labeled paths $P_{\alpha,\beta}, P_{\gamma,\delta}$ and the pseudolinear quadruple $(\alpha,\beta,\gamma,\delta)$ lies on no symmetric labeled cycle in $G_4$.
\end{remark}

\begin{example} Let us show an example of a cycle satisfying condition of Proposition~\ref{p15} with $k>1$.
Let $\pi_1=(1,2,3,4,5,6,7,8)\in S_8$ and let $C=(\pi_1,...,\pi_{12})$ such that
$$
\pi_1(1,2)\pi_2(3,4)\pi_3(1,2)\pi_4(5,6)\pi_5(1,2)\pi_6(7,8)\pi_7
$$
$$
\pi_7(1,2)\pi_{8}(3,4)\pi_{9}(1,2)\pi_{10}(5,6)\pi_{11}(1,2)\pi_{12}(7,8)\pi_1.
$$
%Here $\pi_1(1,2)\pi_2$ means that $\pi_2$ is obtained from $\pi_1$ by transposition of $1$ and $2$.
Here and below $\pi_k(i,j)\pi_l$ means that the permutation $\pi_l$ is obtained from $\pi_k$ by transposition $i$-th and $j$-th elements.
\end{example}

\begin{example} Let us show that a symmetric labeling of a cycle $C$ in $G_n$ is not sufficient for every four points $\{\pi,\varphi,\bar{\pi},\bar{\varphi}\}$ of this cycle forming a pseudolinear quadruple in $(S_n,d_W)$,
where $\bar{\pi},\bar{\varphi}$ are opposite vertices to $\pi,\varphi$ in $C$, respectively.
Indeed, let $\pi_1=(1,2,3,4,5,6)\in S_6$ and let $C=(\pi_1,...,\pi_{8})$ such that
$$
\pi_1(1,2)\pi_2(3,4)\pi_3(1,2)\pi_4(5,6)\pi_5$$
$$
\pi_5(1,2)\pi_{6}(3,4)\pi_{7}(1,2)\pi_{8}(5,6)\pi_{1}.
$$
Consider a quadruple of points $\{\pi_1,\pi_3,\pi_5,\pi_7\}$. For these points holds that $\dsc(\pi_1,\pi_3)=\{(1,2),(3,4)\}$,  $\dsc(\pi_3,\pi_5)=\{(1,2),(5,6)\}$ and $\dsc(\pi_1,\pi_5)=\{(3,4),(5,6)\}$. Using~(\ref{e11}) wee see that neither of the pairs $\{\pi_1,\pi_3\}$, $\{\pi_3,\pi_5\}$, $\{\pi_1,\pi_5\}$ can be a diametrical pair of the pseudolinear quadruple $\{\pi_1,\pi_3,\pi_5,\pi_7\}$.
\end{example}

By $L(\mcp)$ we denote below the set of all labels of the path $\mcp$ and by $\mcp_{\pi,\varphi}$ any of the the shortest paths between $\pi$ and $\varphi$ in $G_n$.
Let $I(W)$ be the set of all elements of the matrix  $W$ which lie above the main diagonal, i.e., $I(W)=\{w_{i,j}\}_{i<j}$.
\begin{proposition}\label{p19}
Let $(S_n,d_W)$ be a metric space, i.e., $w_{i,j}>0$ for all $i<j$, $W$ be a weight  such that for every two different subsets $S_1,S_2\subseteq I(W)$ the relation
\begin{equation}\label{e19}
\sum\limits_{r_i\in S_1}r_i\neq \sum\limits_{r_i\in S_2}r_i,
\end{equation}
holds and let $\pi,\varphi,\bar{\pi},\bar{\varphi}$ be pairwise different points in $S_n$. Then the following conditions are equivalent:
\begin{itemize}
\item[\emph{(i)}] The set $\{\pi,\varphi,\bar{\pi},\bar{\varphi}\}$ form a pseudolinear quadruple in $(S_n,d_W)$ with the diameter $d_W(\pi,\bar{\pi})=d_W(\varphi,\bar{\varphi})$.
\item[\emph{(ii)}] $L(\mcp_{\pi,\varphi})= L(\mcp_{\bar{\pi}, \bar{\varphi}}), \quad
  L(\mcp_{\varphi,\bar{\pi}}) =L(\mcp_{\bar{\varphi},\pi}), \quad L(\mcp_{\pi,\varphi}) \cap L(\mcp_{\varphi,\bar{\pi}})=\varnothing$.
\end{itemize}
\end{proposition}
\begin{proof}
The implication (ii)$\Rightarrow$(i) is almost evident for any $W$.

Let us prove the implication (i)$\Rightarrow$(ii) by contradiction. Without loss of generality suppose first that $L(\mcp_{\pi,\varphi})\neq L(\mcp_{\bar{\pi}, \bar{\varphi}})$. Hence, $\dsc(\pi,\varphi)\neq\dsc(\bar{\pi},\bar{\varphi})$. Using~(\ref{e11}) and ~(\ref{e19}) it follows that the equality $d_W(\pi,\varphi)=d_W(\bar{\pi},\bar{\varphi})$ is impossible.

Suppose that $L(\mcp_{\pi,\varphi}) \cap L(\mcp_{\varphi,\bar{\pi}})\neq\varnothing$ then there exists a label $l$ such that $l\in L(\mcp_{\pi,\varphi}) \cap L(\mcp_{\varphi,\bar{\pi}})$. Since $\mcp_{\pi,\varphi}\cup\mcp_{\varphi,\bar{\pi}}$ is a path connecting $\pi$ and $\bar{\pi}$ in $G_n$ and the label $l$ appears twice in this path,  we have that $l\notin \dsc(\pi,\bar{\pi})$. Clearly, $l\in \dsc(\pi,\varphi)$, $l\in \dsc(\varphi,\bar{\pi})$. Again, using only~(\ref{e11}) we see that the equality $d_W(\pi,\varphi)+d_W(\varphi,\bar{\pi})=d_W(\pi,\bar{\pi})$ is impossible.
\end{proof}

\begin{remark}
There is a simple combinatorial description of all faces of a permutohedron of order $n$: its $k$-faces correspond to ordered partitions of the set $\{1,...,n\}$ into $n-k$ nonempty parts~\cite{Zi95}.
Proposition~\ref{p14} describes pseudolinear quadruples with the diameter $1$.
It is possible to show that every subset of $G_n$ formed from the vertices of $k$-faces, $k\geqslant 2$, contains a cycle $C$, satisfying conditions of Corollary~\ref{c16}.
In other words every $k$-face $2 \leqslant k\leqslant n-2$, contains pseudolinear quadruples in $X$ with diameter strictly less than $\diam X=1$.
\end{remark}

\begin{figure}[ht]
\begin{center}
\begin{minipage}[h]{0.45\linewidth}
\includegraphics[width=1\linewidth]{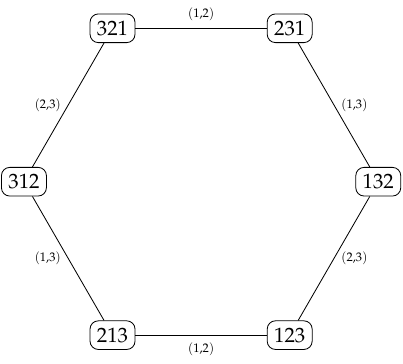}
\caption{The labeled graph $G_3$.}
\label{f2}
\end{minipage}
\hfill
\begin{minipage}[h]{0.45\linewidth}
\includegraphics[width=1\linewidth]{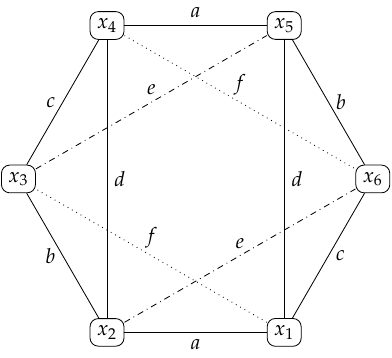}
\caption{The metric space $(X,d)$ with $X=\{x_1,...,x_6\}$.}
\label{f3}
\end{minipage}
\end{center}
\end{figure}

\begin{conjecture}\label{c4.9}
Let $(X,d)$ be a finite metric space such that the following conditions hold:
\begin{itemize}
 \item [(i)] $|X|=n!$;
 \item [(ii)] For every $x \in X$ there is a unique $\bar x\in X$ such that $d(x,\bar x)= \diam X$;
  \item [(iii)] For every two non-diametrical points $x,y$ the set $\{x,y,\bar x, \bar y\}$ form a pseudolinear quadruple;
  \item [(iv)] For every two different points $x,y\in X$ there exists $z\in X$ and a sequence of points $z=p_0,p_1,...,p_k=\bar z$ such that $x,y \in \{p_0,p_1,...,p_k\}$  and for every  $0\leqslant i<j \leqslant k$ the equality
      $$d(p_i,p_j)=d(p_i,p_{i+1})+\cdots+d(p_{j-1},p_{j})$$
      holds, where $k=\binom{n}{2}$. For $k>\binom{n}{2}$ such sequences do not exist.
\end{itemize}
Then $(X,d)$ is isometric to $(S_n,d_W)$ for some weight $W$.
\end{conjecture}

Clearly, for every metric space $(S_n,d_W)$ conditions (i)--(iv) hold. Thus this conjecture asserts that these conditions completely define the structure of $(S_n,d_W)$ up to the weight $W$.

\begin{proof}[for the case n=3]
Let $(X,d)$ be a metric space satisfying conditions (i)--(iv). And let $X=\{x_1,x_2,x_3,x_4,x_5,x_6\}$. By condition (ii), without loss of generality, consider that
$$
d(x_1,x_4)=d(x_2,x_5)=d(x_3,x_6)=\diam X.
$$
It follows from (iii) that
$$
d(x_2,x_4)=d(x_1,x_5)=d, \quad d(x_2,x_6)=d(x_3,x_5)=e,
$$
$$
d(x_1,x_3)=d(x_4,x_6)=f, \quad d(x_1,x_2)=d(x_4,x_5)=a,
$$
$$
d(x_2,x_3)=d(x_5,x_6)=b, \quad d(x_3,x_4)=d(x_6,x_1)=c,
$$
and
\begin{equation}\label{e410}
\diam X=a+d=b+e=c+f,
\end{equation}
see Figure~\ref{f3}.
Clearly, from (iv) for diametrical $x$ and $y$ it follows that $z=x$ and $\bar{z}=y$.
Consider the set of all possible sequences of points connecting the diametrical points $x_1$ and $x_4$ and consisting of  $\binom{3}{2}+1=4$ points:
$$
a_1)\,  x_1,x_6,x_5,x_4;\quad a_2)\, x_1,x_6,x_2,x_4;\quad a_3)\, x_1,x_6,x_3,x_4;
$$
$$
a_4)\, x_1,x_5,x_2,x_4;\quad a_5)\, x_1,x_5,x_3,x_4;\quad a_6)\, x_1,x_5,x_6,x_4.
$$
Without loss of generality, the symmetric case where the points $x_2$ or $x_3$ are at the second place is omitted. Sequences $a_3)$ and $a_4)$ can not satisfy (iv) since they contain consecutive diametrical pairs of points.
Let us consider the sequence $a_1)$. Suppose condition (iv) holds for this case, i.e.,
$$
a+b+c=\diam X.
$$
By~(\ref{e410}) we have
$$
b+c=d,\quad a+c=e, \quad a+b=f.
$$
In this case $(X,d)$ is well-defined, i.e., all triangle inequalities are satisfied.
Using~(\ref{e11}) and Figure~\ref{f2} we see that $(X,d)$ is isometric to $(S_3,d_W)$ with the isometry $\Phi\colon X \to S_3$:
$$
\Phi(x_1)=(1,2,3),\quad\Phi(x_2)=(2,1,3),\quad\Phi(x_3)=(3,1,2),
$$
$$
\Phi(x_4)=(3,2,1),\quad\Phi(x_5)=(2,3,1),\quad\Phi(x_6)=(1,3,2),
$$
and the weight
$$
W=\left(
    \begin{array}{ccc}
      0 & a & b \\
      0 & 0 & c \\
      0 & 0 & 0 \\
    \end{array}
  \right).
$$
Note that this isometry is not necessarily unique.

Consider case $a_2)$. Again, suppose that condition (iv) holds, i.e.,
$$
c+d+e=\diam X.
$$
Using~(\ref{e410}) we have
$$
c+e=a,\quad c+d=b, \quad d+e=f.
$$
In this case $(X,d)$ is isometric to $(S_3,d_W)$, for example, with the isometry \mbox{$\Phi\colon X \to S_3$}:
$$
\Phi(x_1)=(1,2,3),\quad\Phi(x_6)=(1,3,2),\quad\Phi(x_2)=(2,3,1),
$$
$$
\Phi(x_4)=(3,2,1),\quad\Phi(x_3)=(3,1,2),\quad\Phi(x_5)=(2,1,3),
$$
and the weight
$$
W=\left(
    \begin{array}{ccc}
      0 & d & e \\
      0 & 0 & c \\
      0 & 0 & 0 \\
    \end{array}
  \right).
$$
Cases $a_5)$, $a_6)$ are analogous.

Since by the statement of conjecture condition (iv) holds for $(X,d)$, it holds at least for one of the cases $a_1)$, $a_2)$, $a_5)$, $a_6)$ or for their respective symmetric cases which were omitted from consideration. The existence of isometry in each case is shown.
\end{proof}

\section{Conclusion}
Usually, the concept of a metric is associated with the distance between points of a certain space. But in mathematics there are a lot of metrics defined not on points but on completely different mathematical objects. A large number of distances is collected in~\cite{DD16}. Among these distances one can distinguish distances on graphs, matrices, strings, etc. In this work we have considered a metric space the points of which are permutations $S_n$ of the numbers $1,...,n$ with fixed $n$. The introduced metric $d_W$ generalizes not only the well-known Kendall $\tau$ metric but also some another its weighted generalizations.

The paper is devoted to the study of  geometric properties of the space $(S_n,d_W)$.
It is proved that $(S_n,d_W)$ in general case is a pseudometric space and is a metric space if and only if all weights in the strictly upper triangular matrix $W$ are positive. Some basic geometric properties of this space are also described. The observation that the vertex set of a permutohedron of order $n$ coincides with the set of points of the space $(S_n,d_W)$ allowed us to see that the edge-graph $G_n$ of such permutohedron can be used as a convenient tool for studying the space $(S_n,d_W)$. Using the graph $G_n$ we give a criterion which guarantees that some point ``lies between'' another two fixed points from $(S_n,d_W)$ and describe special type four-point subsets of $(S_n,d_W)$ so called ``pseudolinear quadruples''. At the end we formulate a conjecture that characterizes the metric space $(S_n,d_W)$ and prove it for the case $n=3$.

\section{Acknowledgement}
The authors thank the anonymous referees for their remarks which considerably improved this article.

\end{document}